\documentclass[a4paper,12pt]{article} 
\usepackage[T1]{fontenc}
\usepackage[latin2]{inputenc}
\usepackage[a4paper,left=2.5cm,right=2.5cm,top=2.5cm,bottom=2.5cm]{geometry}
\usepackage {times,url,authblk} 
\usepackage {amsmath}
\usepackage {amsfonts}
\usepackage {amssymb}
\usepackage {bm}
\usepackage {graphicx}
\usepackage {color}
\usepackage {import}
\usepackage {float}
\usepackage {tikz}
\usetikzlibrary{decorations}
\usetikzlibrary{patterns}
\usetikzlibrary{decorations.pathmorphing}
\newtheorem{theorem}{Theorem}[section]

\newtheorem{lemma}[theorem]{Lemma}

\def\Box{\raisebox{3pt}{\framebox{\hbox to 3pt{\vbox to 3pt{}}}}}
\newenvironment{proof}{\medskip \noindent{\sc Proof:}}{\quad$\Box$\par\medskip} 
\newenvironment{proof2}[1]{\medskip \noindent{\sc Proof of #1:}}{\quad$\Box$\par\medskip} 
\newtheorem{conjecture}[theorem]{Conjecture}

\newtheorem{claim}[theorem]{Claim}

\newtheorem{question}{Question}
\newcommand{\sfact}{0.28}

\newcommand{\abs}[1]{\left\lvert{#1}\right\rvert}

\newcommand\sq{\mathbin{\text{\scalebox{.84}{$\square$}}}}

\DeclareMathOperator{\opt}{opt}

\begin{document}

\title{The Optimal Pebbling Number of Staircase Graphs}
\author[1,2]{Ervin Gy\H{o}ri\thanks{gyori.ervin@renyi.mta.hu}}
\author[3,4]{Gyula Y. Katona\thanks{kiskat@cs.bme.hu}}
\author[3]{L\'aszl\'o F. Papp\thanks{lazsa@cs.bme.hu}}
\author[1]{Casey Tompkins\thanks{ctompkins496@gmail.com}}
\affil[1]{Alfr\'ed R\'enyi Institute of Mathematics, Budapest, Hungary}
\affil[2]{Department of Mathematics, Central European University, Budapest, Hungary}
\affil[3]{Department of Computer Science and
Information Theory, Budapest University of Technology and Economics, Hungary}
\affil[4]{MTA-ELTE Numerical Analysis and Large Networks Research Group, Hungary}

\maketitle

%General Notes: Pick which spelling to use neighbor or neighbour.   I tried loading amsthm but it conflicts with something.  It would be nice to have the \qedhere command.

%C: Shall we generally replace "m-wide" with "width m" everywhere it occurs? Perhaps that sounds more natural.
%C: Name - "The optimal pebbling number of staircase graphs" ?
%C: Shall we generally replace "m-wide" with "width m" everywhere it occurs? Perhaps that sounds more natural.

\begin{abstract}
Let $G$ be a graph with a distribution of pebbles on its vertices.  A pebbling move consists of removing two pebbles from one vertex and placing one pebble on an adjacent vertex.  The optimal pebbling number of $G$ is the smallest number of pebbles which can placed on the vertices of $G$ such that, for any vertex $v$ of $G$, there is a sequence of pebbling moves resulting in at least one pebble on $v$.   We determine the optimal pebbling number for several classes of induced subgraphs of the square grid, which we call staircase graphs.  
\end{abstract}

\section{Introduction}
Graph pebbling is a game on graphs introduced by Saks and Lagarias. The main framework is the following: A distribution of pebbles is placed on the vertices of a simple graph. A pebbling move removes two pebbles from a vertex and places one pebble on an adjacent vertex.   The goal is to reach any specified vertex by a sequence of pebbling moves.  %This may be viewed as a transportation problem on a graph where the cost of a move is one pebble.  
We begin with some notation which we will need to state our results.

Let $G$ be a simple graph.  We denote the vertex and edge set of $G$ by $V(G)$ and $E(G)$, respectively.     A \emph{pebbling distribution} $P$ is a function from $V(G)$ to the nonnegative integers.  We say that $G$ has $P(v)$ pebbles placed at the vertex $v$ under the distribution $P$.  We say that a vertex $v$ is occupied if $P(v)>0$ and unoccupied otherwise.   The \emph{size} of a pebbling distribution $P$, denoted $\abs{P}$, is the total number of pebbles placed on the vertices of $G$.

Let $u$ be a vertex with at least two pebbles under $P$, and let $v$ be a neighbor of $u$.  A \emph{pebbling move} from $u$ to $v$ consists of removing two pebbles from $u$ and adding one pebble to $v$.  That is, a pebbling move yields a new pebbling distribution $P'$ with $P'(u) = P(u)-2$ and $P'(v) = P(v) + 1$.  We say that a vertex $v$ is \emph{$k$-reachable} under the distribution $P$ if we can obtain, after a sequence of pebbling moves, a distribution with at least $k$ pebbles on $v$.  If $k=1$ we say simply that $v$ is reachable under $P$.   More generally, a set of vertices $S$ is $k$-reachable under the distribution $P$ if, after a sequence of pebbling moves, we can obtain a distribution with at least a total of $k$ pebbles on the vertices in $S$. 

A pebbling distribution $P$ on $G$ is \emph{solvable} if all vertices of $G$ are reachable under $P$.  Similarly, $P$ is \emph{$k$-solvable} if all vertices are $k$-reachable.  A pebbling distribution on $G$ is \emph{$k$-optimal} if it is $k$-solvable and its size is minimal among all of the $k$-solvable distributions of $G$; when $k=1$ we simply say optimal.  Note that $k$-optimal distributions are usually not unique. 
 
The \emph{optimal pebbling number} of $G$, denoted by $\pi_{\opt}(G)$, is the size of an optimal pebbling distribution.   In general, the decision problem for this graph parameter is NP-complete \cite{NPhard}. % The question that is a given distribution and a target vertex is reachable under this distribution is NP-complete even in the class of planar graphs \cite{NPplanar}.

We denote with $P_n$ and $C_n$ the path and cycle on $n$ vertices, respectively.   The \emph{Cartesian product} $G\sq H$ of graphs $G$ and $H$ is defined in the following way: $V(G\sq H)=V(G)\times V(H)$ and $\{(g_1,h_1),(g_2,h_2)\}\in E(G\sq H)$ if and only if $\{g_1,g_2\}\in E(G)$ and $h_1=h_2$ or $\{h_1,h_2\}\in E(H)$ and $g_1=g_2$.

The optimal pebbling number is known for several graphs including paths, cycles \cite{ladder,path1,path2}, caterpillars \cite{caterpillar} and $m$-ary trees \cite{m-ary}.  
The optimal pebbling number of grids has also been investigated.   Exact values were proved for $P_n \sq P_2$ \cite{ladder} and $P_n\sq P_3$ \cite{yerger}.  The question for bigger grids is still open. We gave a construction, which is restated in the next paragraph, for big square grids in \cite{gridnote}. Furthermore, we are going to improve the lower bound on the optimal pebbling number of these grids in a forthcoming paper \cite{gridlower}.  

% C: since we don't make this super precise maybe we should avoid talking about the boundary?
The distribution $P$ which gives the current best upper bound for the optimal pebbling number of the square grid takes groups of seven consecutive diagonals and places pebbles on the middle one  (see Figure~\ref{squaregrid}).   Using these pebbles, it is possible to reach any vertex on any diagonal in the group.

%The distribution  $P$ which gives the current best upper bound for the optimal pebbling number of the square grid places pebbles in every seventh diagonal, and the vertices of seven consecutive diagonals are reachable by using the pebbles of the fourth diagonal (see Figure~\ref{squaregrid}). Therefore, the square grid with this distribution can be decomposed into graphs spanning seven diagonals such that the  distributions induced on each of them are solvable. 

\begin{figure}%[H]
\centering
\begin{tikzpicture}[scale=.45]
\input{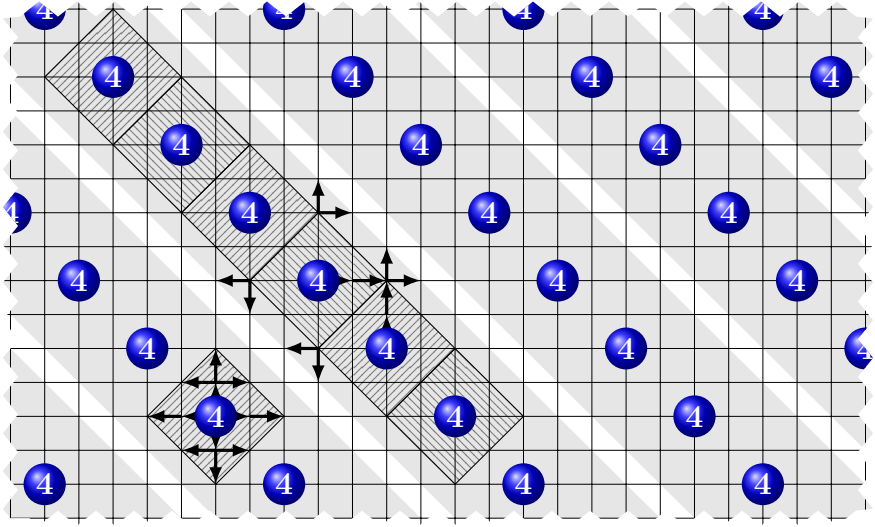}

  %%%
 \clip[decorate, decoration={zigzag,segment length=5mm,amplitude=1mm}](-1,-1) rectangle (24,14);
 %%%%koordinata seged
% \tikzset{koord/.style={draw}};
 %\foreach \x in {0,...,14} \node  (\x) at (-1,\x) {\x};
 %\foreach \x in {0,...,20} \node (\x) at (\x,-1) {\x};
%%%racspontok
\foreach \x in {0,...,25} 
\foreach \y in {0,...,14}
\coordinate (X\x_Y\y) at (\x,\y);
%%%%
%%nagy terulet
%\fill [color=black!20] (-2,2) node (v1) {} -- (0,4) node (v2) {} 
%-- (1,3)-- (2,4) -- (3,3) -- (4,4)-- (5,3) -- (6,4) -- (7,3) -- (8,4)
 % -- (10,2) -- (8,0) 
%-- (7,1) -- (6,0) -- (5,1) -- (4,0) -- (3,1) -- (2,0)-- (1,1) 
%-- (0,0) -- (v1);
%\foreach \r in {-2,0,2,4,6}
%\fill [color=black!20] (\r,2) node (R_\r){} -- +(2,2) -- +(4,0) -- +(2,-2)-- (R_\r);

%%%kis terulet
%\fill [color=black!40] (-1,2) node (v1) {} -- (0,3) node (v2) {} -- (8,3) -- (9,2) -- (8,1) -- (0,1) -- (v1);
%\fill[color=black!20] (8,8) node (w1) {} -- (10,6) -- (8,4) -- (6,6) -- (w1);
%\fill[color=black!10] (2.5,14.5)  -- (-0.5,15) -- (-.5,11.5) -- (11.5,-.5) -- (17.5,-.5) ;
\foreach \i in {-18,-11,-4,3,10,17,24}
\fill[color=black!10]  (\i,15) -- +(6,0) -- +(22,-16) -- +(16,-16);
\draw[pattern=north east lines, pattern color=black!50] (7,2)  -- +(-2,2) -- +(-4,0) -- +(-2,-2) --  (7,2);

\foreach \i in {1,...,3}{
\draw[pattern=north west lines, pattern color=black!50] (4*\i,16-4*\i)  -- +(2,-2) -- +(0,-4) -- +(-2,-2)  --  (4*\i,16-4*\i);
\draw[pattern=north east lines, pattern color=black!50] (4*\i,16-4*\i)  -- +(-2,2) -- +(-4,0) -- +(-2,-2) --  (4*\i,16-4*\i);
}
%% Grid
\draw [very thin, step=1cm] (-1.5,-1.5) grid (25.5,14.5);
%\draw[->,ultra thick] (0,2) -- (-1,2);
%\draw[->,ultra thick] (7,2) -- (7,3);
%\draw[->,ultra thick] (7,2) -- (7,1);
%\draw[->,ultra thick] (7,2) -- (8,2);

%\foreach \p in {0,6} {
%\draw[->,ultra thick] (\p,2) -- (\p,3);
%\draw[->,ultra thick] (\p,2) -- (\p,1);
%}

%%% nyilak
\foreach \s/\t in {
X5_Y2/X5_Y3,
X5_Y2/X6_Y2,
X5_Y2/X4_Y2,
X5_Y2/X5_Y1,
X5_Y3/X5_Y4,
X5_Y3/X6_Y3,
X5_Y3/X4_Y3,
X5_Y1/X5_Y0,
X5_Y1/X6_Y1,
X5_Y1/X4_Y1,
X4_Y2/X3_Y2,
X6_Y2/X7_Y2,
X8_Y8/X8_Y9,
X8_Y8/X9_Y8,
X10_Y6/X11_Y6,
X10_Y6/X10_Y7,
X6_Y6/X6_Y5,
X6_Y6/X5_Y6,
X8_Y4/X8_Y3,
X8_Y4/X7_Y4,
X10_Y4/X10_Y5,
X10_Y5/X10_Y6,
X8_Y6/X9_Y6,
X9_Y6/X10_Y6}
\draw [nyil] (\s) -- (\t);
%%% Pebbles
\foreach \x in {-1,...,7} 
	\foreach \y in  {-3,...,3}  {
	\pgfmathtruncatemacro{\i}{(2*\x+7*\y)}
	\pgfmathtruncatemacro{\j}{(14-2*\x)}
\ifnum \i>-2
\ifnum \i<26 
\node [pebble] (N_\x_\y) at (\i,\j) {$\mathbf{4}$}
\fi
\fi;

}
%% Uj pebble
%\node [pebble2] (plus) at (7,2) {$\mathbf{1}$};

%%% nyilak
\foreach \p in {0,...,5} {
\pgfmathtruncatemacro{\q}{(\p+1)}
%\draw [<-, densely dotted, ultra thick]   (N_\p) -- (N_\q);
}
%\draw [->,  densely dotted, ultra thick]   (plus) -- (N_6);
\end{tikzpicture}
\label{squaregrid}
\caption{Solvable distribution of the square grid.}
\end{figure}

We conjecture that $P$ is optimal, however we do not know a proof for this. Can we at least show that the induced distributions on these smaller graphs are optimal?  If this was not the case, it would refute the conjecture.  These considerations provide the main motivation for the present paper.

%Furthermore, if one of them is not optimal, then it immediately refutes the conjecture. % about square grid. 
%This is the main motivation for the present paper.
%The optimal pebbling number of these graph can be determined using the technics invented in \cite{ladder}.

%It is easy to see that if $P$ is optimal on the squaregrid, then these smaller parts are also optimal. So if the optimal pebbling number of such a subgraph $S$ is less than $\abs{P_S}$, then the upper bound of \cite{gridnote} is not sharp, in contrast our beliefs. 

We investigate a family of graphs which we call \emph{staircase graphs}. These graphs are connected induced subgraphs of the square grid. The width seven instances correspond to the groups of seven diagonals discussed above.   

Let $SG_{n,n}=P_n\sq P_n$ and let $SG = P_{\infty} \sq P_{\infty}$ be the infinite square grid where $P_{\infty}$ is the doubly infinite path with vertex set $\mathbb{Z}$ and edge set $\{\{i,i+1\}:i \in \mathbb{Z}\}$.  

For any $k \in \mathbb{Z}$, we  say that $D^{+}_k=\left \{ \{i,j\} \in V(SG): i-j=k \right\}$ is a \emph{positive diagonal} of $SG$.
Similarly we define the \emph{negative diagonal}: $D^{-}_k=\left  \{ \{i,j\} \in V(SG): i+j=k \right\}$.    A staircase graph will be defined in terms of the intersection of a set of consecutive positive diagonals in $SG$ with a set of consecutive negative diagonals.  When the number of diagonals taken in each direction is odd, there will be two nonisomorphic graphs to consider.
%We should reference a figure here to show the different staircase graphs -Casey

If $m$ is odd, let $S'_{m,n}$ be the graph induced by the vertex set $\left( \cup_{j=1}^m D^-_{j}\right)\cap\left (\cup_{i=1}^n D^+_{i}\right ) $, and let $S_{m,n}$ be the graph induced by $\left( \cup_{j=1}^m D^-_{j}\right) \cap \left (\cup_{i=0}^{n-1} D^+_{i}\right )$. 

If $m$ is even, let $S_{m,n}$ be the graph induced by the vertex set $ \left( \cup_{j=1}^m D^-_{j}\right)\cap\left (\cup_{i=1}^n D^+_{i}\right )$. In this case we have only one isomorphism class.

%In the $m \equiv 1 \mod 4$ case we swap the definitons of $S_{m,n}$ and $S'_{m,n}$.

Note that $S'_{m,n}\cong S_{m,n}$ in case when $n$ is even.   We also remark that $S_{m,n}\cong S_{n,m}$; nevertheless, we say that the first and the second parameters are the \emph{width} and the \emph{length} of the staircase graph, respectively, and generally assume that $n\geq m$.  We will refer to the graphs $S_{m,n}$ and $S_{m,n}'$ as \emph{m-wide staircase graphs}.

\begin{figure}
\centering
\includegraphics[scale=\sfact]{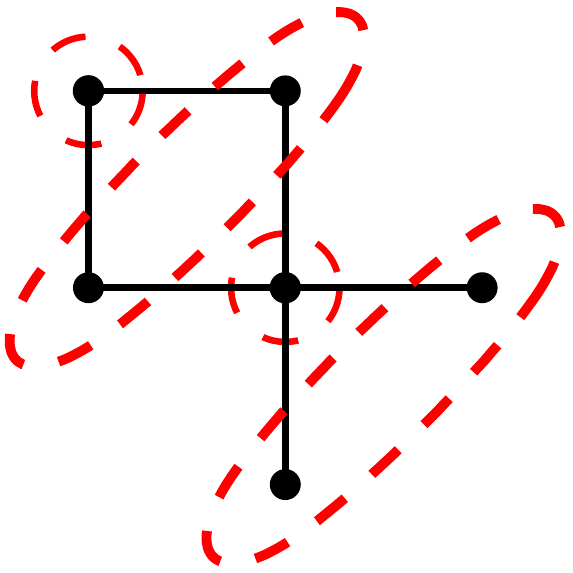}
\label{1slash}
\caption{Slashes in the three-wide staircase graph.}
\end{figure}
For simplicity, we call a nonempty intersection of a staircase graph and a positive diagonal a \emph{slash} (See Figure \ref{1slash} where each dashed ellipse is a slash).%Then the length of $G$ is the number of slashes contained in it. 

The optimal pebbling number of ladders is proved by a technique based on induction and cutting. We use this technique for narrow staircase graphs and extend it for wider ones.

\section{Results}

\subsection{Three-wide staircases}

The results in the 3 and 4-wide cases depend on the value of $n$ modulo 4. Therefore we write $4k+r$ instead of $n$, where $r\in\{0,1,2,3\}$. 
\begin{theorem}
When $4k+r\geq 2$, then

%$$\pi_{\opt}(S_{3,n})=\begin{cases}\left\lceil \frac{3n}{4}\right\rceil & \text{if } n\equiv 0 \mod 4\\
%					 \left\lceil \frac{3n}{4}\right\rceil+1 & \text{if } n\equiv 1 \mod 4\\
%					 \left\lceil \frac{3n}{4}\right\rceil+2 & \text{if } n\equiv 2 \mod 4\\
%					 \left\lceil \frac{3n}{4}\right\rceil+3 & \text{if } n\equiv 3 \mod 4
%		\end{cases}$$			
%

\begin{align*}
\pi_{\opt}(S_{3,4k+r})&=3k+r, \\
\pi_{\opt}(S'_{3,4k+r})&=\begin{cases}
3k+2 & \text{if }r=3\\
3k+r & \text{otherwise}.
\end{cases}
\end{align*}

\label{theorem3stairs}

\end{theorem}

\begin{figure}%[H]
\centering
\scalebox{\sfact}{\input{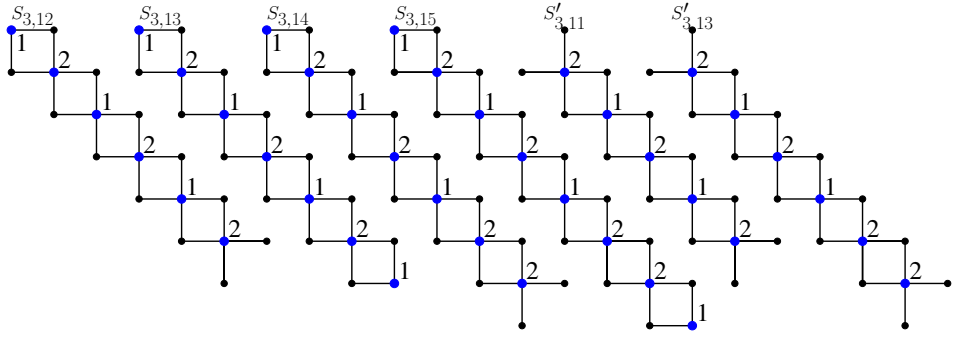}}

\caption{Optimal distributions of $S_{3,n}$ and $S'_{3,n}$ when $n\geq 5$.}
\label{Sopt}
\end{figure}

We also note that $\pi_{\opt}(S_{3,1}) = 1$ and $\pi_{\opt}(S_{3,1}') = 2$, since these cases may appear in our induction. Before we present the proof we establish some lemmas. Each of them utilizes the one preceding it, and the first one strongly relies on the following theorem from \cite{ladder}.

\begin{theorem}[\cite{ladder}]
A 2-optimal distribution of the $n$-vertex path contains $n+1$ pebbles.
\label{2optimal}
\end{theorem}

\begin{lemma}
\label{existsnot2reachable}
Let $G$ be $S_{m,n}$ or $S'_{m,n}$ and suppose there is a solvable distribution $P$ on $G$ such that $|P|<n+1$. Then there is a slash in $G$ which is not 2-reachable under $P$. 
\end{lemma}

%For narrow staircases we utilize a proving technique based on induction and cutting. The same techniqe is used for ladders in \cite{ladder}. 

We are going to use the \emph{collapsing technique} introduced in \cite{ladder}. Let $G$ and $H$ be simple graphs. We say that $H$ is a \emph{quotient} of $G$, if there is a surjective mapping $\phi: V(G)\rightarrow V(H)$ such that $\{h_1,h_2\}\in E(H)$ if and only if there are $g_1,g_2\in V(G)$, where $\{g_1,g_2\}\in E(G)$, $h_1=\phi(g_1)$ and $h_2=\phi(g_2)$. We say that $\phi$ collapses $G$ to $H$, and if $P$ is a pebbling distribution on $G$, then the \emph{collapsed distribution} $P_{\phi}$ on $H$ is defined in the following way:
$P_{\phi}(h)=\sum_{g\in V(G)| \phi(g)=h}P(g)$.

\begin{proof}
Let $\phi$ be a mapping which maps each slash in $G$ to a vertex in $P_n$ in such a way that consecutive slashes are mapped to adjacent vertices. %The collapsed graph $\phi(G)$ which we get is a path. 
It is easy to see that if a slash was 2-reachable under $P$, then its collapsed image is also 2-reachable under $P_{\phi}$. 
$P$ and $P_{\phi}$ have less than $n+1$ pebbles. Therefore, Theorem~\ref{2optimal} yields that $P_{\phi}$ is not 2-solvable, therefore there is a vertex in $P_n$ which is not 2-reachable under $P_{\phi}$ and the corresponding slash is not 2-reachable under $P$.
\end{proof}

\begin{lemma}
Let $G$ be $S_{m,n}$ or $S'_{m,n}$  and suppose there is a solvable distribution $P$ on $G$ such that $|P|<n-1$. Then there is slash in $G$ which is neither the first nor the last slash and is not 2-reachable under $P$. 
\label{not2reachable}
\end{lemma}

\begin{proof}
%We develop the previous proof. %If $n\geq 8$, then $\pi_{\opt}(G)\leq n-2$. 
If there are only two slashes which are not 2-reachable under $P$, then we put a pebble on a vertex in both slashes and all slashes are 2-reachable with less than $n+1$ pebbles. This contradicts Lemma \ref{existsnot2reachable}. Therefore, there are at least three slashes which are not 2-reachable under $P$. One of them is neither the first nor the last slash.
\end{proof}

\begin{lemma}
If slash $k$ is not 2-reachable under a distribution $P$ and it is not the first or the last slash, then there is no possible pebbling move between either slash $k-1$ and $k$ or slash $k$ and $k+1$.
\label{cancutlemma}
\end{lemma}

%\begin{proof}
%If a slash is not 2-reachable, then no pebble can be moved away from this slash. Thus, slash $k$ divides the graph to two parts. The pebbles placed at slashes numbered below $k$ cannot be moved to slashes numbered with higher numbers and vice versa.  Therefore if we can move $x$ pebbles to slash $k-1$ and $y$ pebbles to slash $k+2$, then we can do these operations simultaneously. Thus, it is not possible that slashes $k-1$ and $k+1$ are both 2-reachable because this yields that we can move a pebble from each of them to slash $k$ which contradicts our assumption. So slash $k$ cannot give a pebble to its neighbours and cannot receive a pebble from at least one of them.    
%\end{proof}

\begin{proof} %[rewording of pf with corrected details]
Since slash $k$ is not 2-reachable, it is impossible to make a pebbling move starting at a vertex from slash $k$.  It follows that the pebbling distributions which can be obtained on the slashes numbered less than $k$ do not depend on the pebbling distribution on slashes numbered greater than $k$ and vice versa.  If slash $k$ could be reached from a pebbling move from both slashes $k-1$ and $k+1$, then these moves could be performed independently and it would follow that slash $k$ is 2-reachable, a contradiction.
\end{proof}

If there is no possible pebbling move between slashes $k$ and $k+1$ then, after the deletion of the edges between them, each vertex of the resulting graph is still reachable under $P$. Hence $P$ induces a solvable distribution on each connected component. We will use this fact during later proofs.

\begin{figure}%[H]
\centering
\scalebox{\sfact}{\input{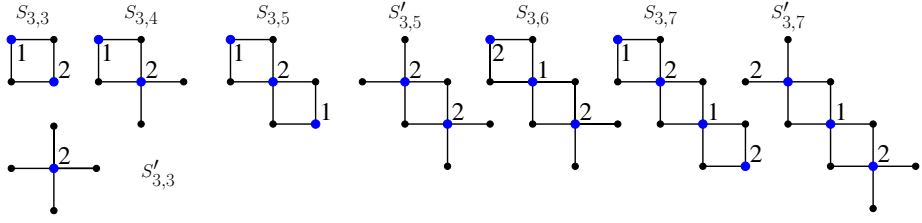}}

\caption{Optimal distributions of small $S_{3,n}$ and $S'_{3,n}$ graphs.}
\label{S3base}
\end{figure}

\begin{claim}
\label{lowerbound}
Let $G$ be $S_{m,n}$ or $S'_{m,n}$, and assume that there is a solvable distribution on $G$ with less than $n-1$ pebbles. Then, there is a $c$ such that $1\leq c <n$ and:
\begin{align*}
\pi_{\opt}(S_{m,n})\geq \pi_{\opt}(S_{m,c})+\pi_{\opt}(S_{m,n-c})& \text{ if } c \text{ is even}\\ 
\pi_{\opt}(S_{m,n})\geq \pi_{\opt}(S_{m,c})+\pi_{\opt}(S'_{m,n-c})& \text{ if } c \text{ is odd}\\
\pi_{\opt}(S'_{m,n})\geq \pi_{\opt}(S'_{m,c})+\pi_{\opt}(S'_{m,n-c})& \text{ if } c \text{ is even}\\ 
\pi_{\opt}(S'_{m,n})\geq \pi_{\opt}(S'_{m,c})+\pi_{\opt}(S_{m,n-c})& \text{ if } c \text{ is odd} .
\end{align*}
If $m$ is even, we have only one inequality:
\begin{equation*}
\pi_{\opt}(S_{m,n})\geq \pi_{\opt}(S_{m,c})+\pi_{\opt}(S_{m,n-c}).
\end{equation*}
\end{claim}

\begin{proof}
An optimal pebbling distribution $P$ has at most $n-1$ pebbles and so we can apply Lemma~\ref{not2reachable} and Lemma~\ref{cancutlemma}. Hence there is a slash $k$, which is neither the first nor the last, such that we cannot move a pebble between slashes $k$ and $k+1$ or $k$ and $k-1$. We choose $c$ to be $k$ in the first scenario and $k-1$ in the second. Then, we delete the edges between these slashes and obtain two smaller graphs with solvable distributions. The sum of the sizes of these two distributions is the same as $|P|$. The size of a solvable distribution is always at least as big as the optimal pebbling number. This gives us a lower bound.
The types of the two resulting graphs depend on the type of the original graph and the parity of $c$. 
\end{proof}

\begin{proof2}{Theorem \ref{theorem3stairs}}
The upper bound comes from solvable distributions shown in Figure \ref{Sopt}.

We prove the lower bounds for $S_{3,n}$ and $S'_{3,n}$ simultaneously by induction on $n$.
The cases where $n\leq 7$ are shown in Figure \ref{S3base}. It can be checked by hand or computer that these are optimal distributions.

Therefore, we will assume that $n\geq 8$ and the lower bound is true for values smaller than $n$. In this case we can apply Claim \ref{lowerbound}. 
Unfortunately, we do not know the value of $c$, hence we have to check all the possible values to prove a lower bound. The formulas which we want to prove and use as an inductive hypothesis are quite complicated and depend on the value of $n$ modulo 4. Therefore, we have several cases.  We will see that some of these cases are not possible if $P$ is optimal. 
%This means, that there are several adjacent slashes ($c$ and $c+1$) such that a pebble can be moved between them in any optimal distribution. 
  %C -  I found the preceding sentence confusing but I don't see a good way to phrase it.
  
First, consider $S_{3,n}$. We remind the reader that $n=4k+r$, where $r\in\{0,1,2,3\}$. %, where $n=4k+r$ and $0\leq r<4$.
We cut between the slashes $c$ and $c+1$, where $c=4l+q$ and $q\in \{0,1,2,3\}$. Similarly, assume that $n-c =4j+p$ where $p\in\{0,1,2,3\}$.  Note that $l+j$ equals $k$ when $q+p<4$ and $k-1$ otherwise. Furthermore, the parity of $c$ and $q$ are the same. When $q$ is even:
%One of the gotten graphs is always $S_{3,c}$, on the other hand the other one can be either $S_{n-c}$ or $S'_{n-c}$. It is $S_{n-c}$ if $c$ is even and $S'_{n-c}$ if its is odd. Therefore if $q$ is even:
\begin{equation*}
\begin{split}
\pi_{\opt}(S_{3,n})&\geq\pi_{\opt}(S_{3,c})+\pi_{\opt}(S_{3,n-c})= 3l+q+3j+p=3(l+j)+q+p\\
&=\begin{cases}
3k+r & \text{if }q+p<4\\
3(k-1)+4+r & \text{if }q+p\geq 4.
\end{cases}
\end{split}
\end{equation*}
  Note that the second case in the inequality is not possible because it would give a bigger lower bound on $\pi_{\opt}(S_{3,n})$ than the known upper bound.
 
When $q$ is odd and $p\neq 3$, we obtain the same estimate on $\pi_{\opt}(S_{3,n})$ except in the case when $j=0$ and $p=1$ where the right-hand side is larger by 1 which is not possible for an optimal $P$.  
If $q$ is odd and $p=3$, then $p+q=4+r$, therefore:
\begin{equation*}
\pi_{\opt}(S_{3,n})\geq\pi_{\opt}(S_{3,c})+\pi_{\opt}(S'_{3,n-c}) = 3l+q+3j+p-1=3(l+j)+3+r=3k+r.
\end{equation*}

%This case is also not possible.

Now we consider $S'_{3,n}$. We may assume $n$ is odd because $S'_{3,n}\cong S_{3,n}$ when $n$ is even.
If $r=3$, then:
\begin{align*}
&\pi_{\opt}(S'_{3,4k+3})\geq\pi_{\opt}(S'_{3,4l})+\pi_{\opt}(S'_{3,4j+3})\geq 3l+3j+2=3k+2 & &\text{if } q=0\\
&\pi_{\opt}(S'_{3,4k+3})\geq\pi_{\opt}(S'_{3,4l+1})+\pi_{\opt}(S_{3,4j+2})\geq 3l+1+3j+2=3k+3 & &\text{if } q=1\\
&\pi_{\opt}(S'_{3,4k+3})\geq\pi_{\opt}(S'_{3,4l+2})+\pi_{\opt}(S'_{3,4j+1})\geq 3l+2+3j+1=3k+3 & &\text{if } q=2\\
&\pi_{\opt}(S'_{3,4k+3})\geq\pi_{\opt}(S'_{3,4l+3})+\pi_{\opt}(S_{3,4j})\geq 3l+2+3j=3k+2 & &\text{if } q=3.
\end{align*}
Finally, when $r=1$:
\begin{align*}
&\pi_{\opt}(S'_{3,4k+1})\geq\pi_{\opt}(S'_{3,4l})+\pi_{\opt}(S'_{3,4j+1})\geq 3l+3j+1=3k+1 & &\text{if } q=0\\
&\pi_{\opt}(S'_{3,4k+1})\geq\pi_{\opt}(S'_{3,4l+1})+\pi_{\opt}(S_{3,4j})\geq 3l+1+3j=3k+1 & &\text{if } q=1\\
&\pi_{\opt}(S'_{3,4k+1})\geq\pi_{\opt}(S'_{3,4l+2})+\pi_{\opt}(S'_{3,4j+3})\geq 3l+2+3j+2=%3(k-1)+4=
3k+1 & &\text{if } q=2\\
&\pi_{\opt}(S'_{3,4k+1})\geq\pi_{\opt}(S'_{3,4l+3})+\pi_{\opt}(S_{3,4j+2})\geq 3l+2+3j+2=%3(k-1)+4=
3k+1 & &\text{if } q=3.
\end{align*}

We have checked each case and the resulting lower bound is at least as strong as the desired one,  hence the proof is complete.
\end{proof2}

\subsection{Four-wide staircases}

In this case the proof is simpler because $S'_{4,n}$ is isomorphic to $S_{4,n}$. 

\begin{theorem}
\begin{equation*}
\pi_{\opt}(S_{4,4k+r})=3k+r
%\pi_{\opt}(S_{4,n})=\left\lfloor\frac{3n}{4}\right\rfloor+(n\mod 4)
\end{equation*}
except in the following small cases:
$\pi_{\opt}(S_{4,1})=2, 
\pi_{\opt}(S_{4,2})=3.$
\end{theorem}

\begin{figure}[H]
\centering
\scalebox{\sfact}{\input{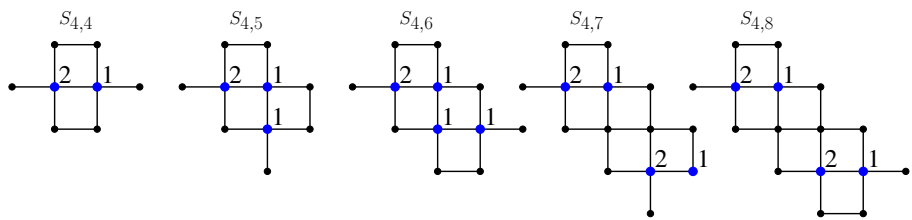}}

\caption{Optimal distributions of small $S_{4,n}$ graphs. For the $n\leq 3$ cases refer to the previous figures.}
\label{S4opt}
\end{figure}

\begin{proof}
The optimal distributions of small graphs are shown in Figure \ref{S4opt}. For larger $n$, we obtain a solvable distribution of $S_{4,n}$ in the following way: We partition $S_{4,n}$ into $k-1$ copies of $S_{4,4}$ and one copy of $S_{4,4+r}$ for some $r$, and we use the optimal distributions shown in Figure \ref{S4opt} in each of the parts. This gives the required upper bound.

To obtain the lower bound for $n\geq 8$, we assume that the statement is true for all values below $n$. We have $\pi_{\opt}(S_{4,n})\leq n-2$ by the upper bound so 
Claim \ref{lowerbound} can be applied yielding
\begin{equation*}
\pi_{\opt}(S_{4,n})\geq \pi_{\opt}(S_{4,c})+\pi_{opt}(S_{4,n-c}) \ge 3l+q+3j+p\geq 3k+r,
\end{equation*}
where $n=4k+r$, $c=4l+q$, $n-c=4j+p$ and $r,q,p\in\{0,1,2,3\}$. 
\end{proof}

\subsection{Five-wide staircases}

In this case we again must distinguish between the graphs $S_{5,n}$ and $S_{5,n}'$ when $n$ is odd. Fortunately, we are in a simpler situation than in the three-wide case because the formula optimal pebbling number of both these graphs turns out to be the same. The formula depends on the value of $n$ modulo 5, so in this section let $n=5k+r$ where $r\in\{0,1,2,3,4\}$.
\begin{theorem}

%$$\pi_{\opt}(S_{5,n})=\pi_{\opt}(S'_{5,n})=\left\lceil \frac{4n}{5}\right\rceil+(n\mod 5)	$$			
		
$$\pi_{\opt}(S_{5,5k+r})=\pi_{\opt}(S'_{5,5k+r})=4k+r,	$$			 
except for $n\in \{1,2,3,7\}$. $\pi_{\opt}(S_{5,3})=\pi_{\opt}(S'_{5,3})=4$ and $\pi_{\opt}(S'_{5,7})=7$. 
%Remaining two cases?? 1 wide is 5 isolated vertices, 5 wide is a path. 

\end{theorem}

\begin{proof}
The smaller cases can be seen in Figure \ref{S5opt}. Solvable distributions with the given sizes can be constructed with the proper concatenation of the optimal distributions of $S_{5,4}$, $S_{5,5}$, $S_{5,6}$, $S_{5,7}$, $S_{5,8}$, $S_{5,9}$, $S'_{5,5}$ and $S'_{5,9}$. For large $n$ use many copies of $S_{5,5}$ and $S'_{5,5}$ in the middle and suitably extend it on the ends with the other distributions. 

We can still apply Claim \ref{lowerbound} when $n\geq 10$. The formula is the same for $S_{5,n}$ and $S'_{5,n}$, therefore we introduce the notation $S^*_{5,n}$ to denote either $S_{5,n}$ or $S'_{5,n}$. We can formalize the statement of Claim~\ref{lowerbound} in one inequality:
\begin{equation*}
\pi_{\opt}(S^*_{5,n})\geq \pi_{\opt}(S^*_{5,c})+\pi_{\opt}(S^*_{5,n-c})=4l+q+4j+p\geq 4k+r,
\end{equation*}	
where $n=5k+r$, $c=5l+q$, $n-c=5j+p$ and $r,q,p\in\{0,1,2,3,4\}$. 
\end{proof}

\begin{figure}%[H]
\centering
\scalebox{\sfact}{\input{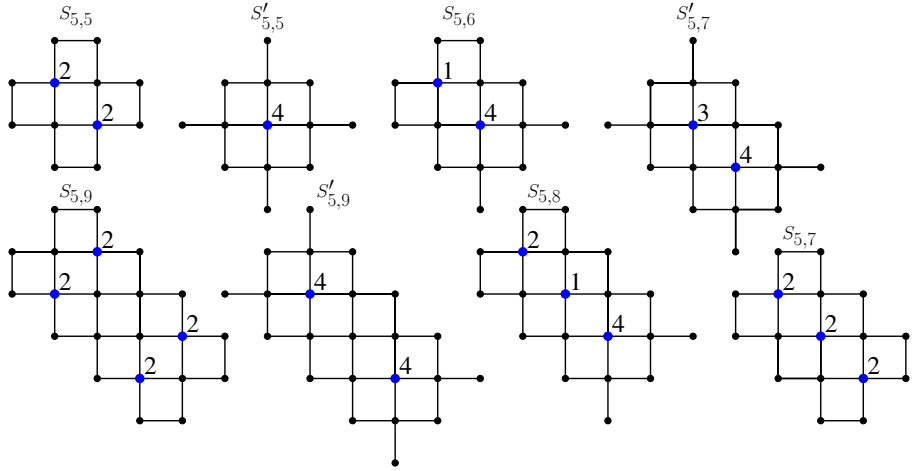}} %original scale was 0.55

\caption{Optimal distributions of small $S_{5,n}$ graphs. For the cases where $n\leq 4$ refer to the previous figures.}
\label{S5opt}
\end{figure}

\subsection{Six-wide staircases}

To handle this case and the seven-wide case we need some additional tools. We start by recalling some known results about the path.
\begin{lemma}[\cite{ladder}]
A 2-optimal distribution of $P_n$ consists of  \emph{prime segments} separated by single unoccupied vertices, where a prime segment is a subpath of one of two possible types. Either all but one of the vertices have one pebble and one vertex has two pebbles or there are three consecutive vertices with zero, four and zero pebbles, in that order, and the remaining vertices having one pebble.  
\label{primesegments}
\end{lemma}

A corollary of this lemma is that under a 2-optimal distribution no vertex is 5-reachable, and if a vertex is 3-reachable, then it has to contain 4 pebbles. The neighbours of a vertex $v$ having 4 pebbles can get pebbles only from $v$. Similarly, each of the neighbours of a vertex $u$ containing exactly 2 pebbles can get only one pebble and only from $u$. Using these facts we obtain the following statement for staircases.  
%It is easy to check these statements on prime segments, and the vertices separating them guarantees that two or more prime segments do not help each other violating these rules. They imply the following statement for staircases: 
%C - I think these corrolaries are pretty obvious and explaining them might make it more confusing.

\begin{lemma}
\label{uniquecor}
Let $G$ be a staircase graph with $n$ slashes. If $P$ is a pebbling distribution on $G$ with $n+1$ pebbles, such that each slash is 2-reachable under $P$, then:
\begin{itemize}
\item Each slash contains at most 4 pebbles.
\item If a slash is 3-reachable, then it contains 4 pebbles.
\item if a slash has at least one pebble, then at most one pebble can be moved to that slash with a pebbling move.
%C -  above \item is new.  
\item If a slash contains 4 pebbles, then the adjacent slashes have no pebbles, furthermore they can get pebbles only from this slash.
\item If a slash contains 2 pebbles, then a pebbling move can only be performed ending in that slash if a pebbling move was first performed beginning on that slash.
\end{itemize}
\end{lemma}

\begin{proof}
We simply apply the collapsing function defined in the proof of Lemma \ref{existsnot2reachable}.  The collapsed distribution on $P_n$ is 2-optimal so Lemma \ref{primesegments} applies.  All of the statements then follow from easily observable facts about prime segments.
%$\phi(G)=P_n$ and $\phi(P)$ is a $2$-optimal distribution of $P_n$. We can apply Lemma \ref{primesegments} to $\phi(G)$, and obtain the analogues statements for the vertices of $\phi(G)$. The preimage of a vertex is a slash and $P$ puts as many pebbles at a slash as $\phi(P)$ puts at its image. Moreover, if a slash is $k$-reachable in $G$, then its image is also $k$-reachable in $\phi(G)$.    
\end{proof}

\begin{theorem}
\label{6widetheorem}
$$\pi_{\opt}(S_{6,n})=n,$$			
		 
except for $n\in \{1,2,3,4,8,9\}$.
$\pi_{\opt}(S_{6,3})=\pi_{\opt}(S_{6,4})=5$,
$\pi_{\opt}(S_{6,8})=9$ and
$\pi_{\opt}(S_{6,9})=10$.
\end{theorem}

To prove the lower bound we cannot use exactly the same method which we used for the narrower staircases because Lemma \ref{not2reachable} does not apply when $|P|\geq n-1$. We can overcome this difficulty with the next lemma.

\begin{figure}%[H]
\centering
\scalebox{\sfact}{\input{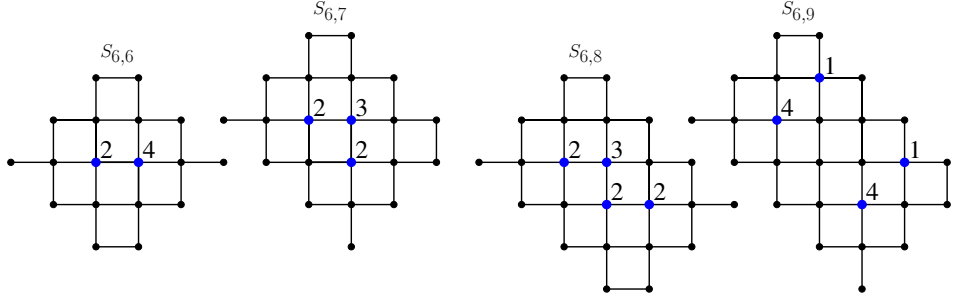}}

\caption{Optimal distributions of small $S_{6,n}$ graphs. For $n\leq 5$ cases check the previous figures.}
\label{S6base}
\end{figure}

\begin{lemma}
There is no solvable distribution $P$ on $S_{6,n}$ with size at most $n$, such that each inner slash is 2-reachable.
\label{lemmacont}
\end{lemma}

\begin{figure}%[H]
\scalebox{\sfact}{\input{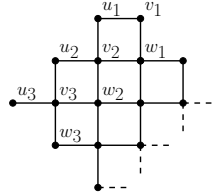}}
\centering
\caption{Labelling of the vertices}
\label{beginslash6}
\end{figure}

\begin{proof}
Indirectly assume that such a $P$ exists. The size of $P$ can be either $n-1$ or $n$, otherwise an inner slash must be not 2-reachable according to Lemma \ref{not2reachable}. Without the loss of generality, we assume that the first slash is not 2-reachable. %Because if one of the is 2-reachable then put one pebble at the other one and all slashes becomes 2-reachable with $n$ pebbles.

%Consider the first slash. It is not 2-reachable under $P$, so it can not have two or more pebbles. 
Since $P$ is solvable, we can reach all vertices of the first slash. If the first slash has exactly one pebble, then we need to move another pebble to reach its unoccupied vertices. This means that it is 2-reachable, therefore the first slash cannot have a pebble. 

Put a pebble on the first slash and, in the case of $|P|=n-1$,  another one at the last slash. This results a new distribution $P'$ and each slash is 2-reachable under it. Therefore $|P'| = n+1$ and, if we collapse the graph into $P_n$ with $\phi$, then the collapsed distribution $P'_{\phi}$ is a 2-optimal distribution. 

We will require some additional notation. Vertex names are shown in Figure \ref{beginslash6}. Now we exploit Lemma \ref{uniquecor}.

The first slash contains exactly one pebble under $P'$. By Lemma \ref{uniquecor} (points 2 and 4) the second slash cannot be 3-reachable under $P'$. So it has at most two pebbles.
 $v_3$ has to be 2-reachable under $P$ and $P'$ because $u_3$ can be reachable only from this vertex. The reachability of $u_1$ requires that $v_1$ or $v_2$ is also 2-reachable.

If the second slash has two pebbles, then it is clear from the last statement of Lemma \ref{uniquecor} that at most one vertex in that slash is 2-reachable.   This contradicts that $v_3$ and one of $v_1$ or $v_2$ are 2-reachable.

% cannot get a pebble from other slashes according to the last statement of Lemma \ref{uniquecor}, but these two pebbles are not enough to make make $v_3$ and one of $v_1$, $v_2$ $2$-reachable.

When the second slash has exactly one pebble, then the third statement of Lemma \ref{uniquecor} implies at most one vertex is $2$-reachable in the second slash, which is not enough.

So neither the first nor the second slash contain a pebble under $P$, but the reachability of the first slash requires that the third slash is 4-reachable under $P'$.  According to Lemma \ref{uniquecor} the third slash must contain exactly 4 pebbles.  Every vertex in the first three slashes must be reachable using only the 4 pebbles on slash three, but this is impossible (in particular we cannot reach both $u_3$ and $v_1$).
%There are no pebbles at the fourth slash and it can not get a pebble from the fifth slash, therefore it is not worth to move the pebbles of the third slash among the vertices of the fourth slash when we try to reach a vertex of the first or second slash.
%Therefore there are 4 pebbles on the third slash and we have to reach all vertices of the first three slashes using only these pebbles under $P$.  It does not matter how we distribute the 4 pebbles, it can not be done. Therefore we have a contradiction here and our assumption was false.
\end{proof}

\newpage

\begin{proof2}{Theorem \ref{6widetheorem}}
For small optimal distributions see Figure \ref{S6base}.  Solvable distribution with  $n$ pebbles in cases when $n\geq 10$ can be created by combining the optimal distributions of $S_{6,5}$, $S_{6,6}$ and $S_{6,7}$. 

%Unfortunately now we can not apply Lemma \ref{not2reachable}, because $\pi_{\opt}(S_{6,n})>n-2$.  In contrast, we can by-pass this lemma. We need some more work to do and use some more result of $\cite{ladder}.$

Assume that the lower bound is proved for each integer less than $n$. 
Let $P$ be an optimal distribution. We know that the size of $P$ is at most $n$. By Lemma \ref{lemmacont} we have that some inner slash is not 2-reachable. Then we can apply %Lemma \ref{cancutlemma} and break the graph to two parts between slash $c$ and $c+1$.
Claim \ref{lowerbound} and the induction hypothesis, which gives that $|P|\geq c+n-c=n$. This completes the proof.
\end{proof2}
\subsection{Seven-wide staircase graphs}

\begin{lemma}
There is no solvable distribution $P$ on $S_{7,n}$ with size at most $n$, such that each inner slash is 2-reachable.
\label{lemmacont7}
\end{lemma}

%\begin{proof}
%We can tell the same what we told in the six wide case. %ine in the proof of the similar Lemma \ref{lemmacont}.
%$S_{6,n}$ is an induced subgraph of $S_{7,n}$ and $S'_{7,n}$. There are no ''shortcuts'' between the vertices of a copy of $S_{6,n}$ through the remaining vertices of $S_{7,n}$ or $S'_{7n}$ Roughly speaking, the additional vertices does not decrease the distances. 
%Therefore the unsolvable task of 4 pebbles distribution is even harder, hence it is still impossible to do.  
%\end{proof}

\begin{proof}
Assume the contrary.  We have a solvable distribution $P$ on $S_{7,n}$ with size $n$, such that each inner slash is 2-reachable.

Now we collapse $S_{7,n}$ to $S_{6,n}$ by mapping the first and the third negative diagonal of $S_{7,n}$ to one diagonal of $S_{6,n}$.  For an example see Figure \ref{7to6}.
%We create a distribution $P'$ on $S_{6,n}$ in the following way: We start with $P$ on $S_{7,n}$ and remove the first negative diagonal. If a removed vertex has pebbles under $P$, then we place these pebbles to the closest vertex of the slash where the removed vertex belonged. For an example see Figure \ref{7to6}. 

\begin{figure}[H]
\centering
\scalebox{\sfact}{\input{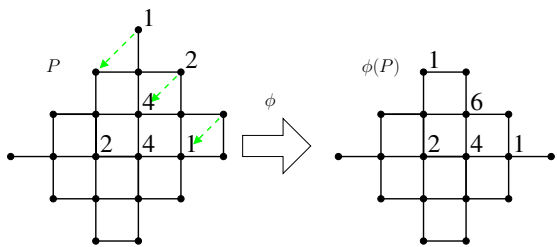}}
\caption{Collapsing $S_{7,6}$ to $S_{6,6}$.}
\label{7to6}
\end{figure}

Each pebbling sequence which is executable under $P$ determines an executable pebbling sequence in $P_{\phi}$ which moves the same amount of pebbles to the slashes and moves at least as many pebbles to the same vertices except for the deleted ones.
%The process is simple: The modified sequence uses the third positive diagonal instead of the first positive diagonal.
%$P_{\Phi}$ is solvable and each inner slash is $2$-reachable.
Therefore, $P_{\phi}$ is solvable and each inner slash is $2$-reachable. This contradicts Lemma \ref{lemmacont}.    
\end{proof}

\begin{theorem}
Let $S^*_{7,n}$ be $S_{7,n}$ or $S_{7,n}'$, then
$$n+1\leq \pi_{\opt}(S^*_{7,n})\leq n+3.$$
The lower bound is sharp for graphs $S_{7,5}$, $S_{7,6}$, $S_{7,7}$, $S_{7,8}$ and every $S'_{7,n}$ where $n\equiv 3\bmod 4$.		
\end{theorem}

\begin{figure}[htb]
\centering
\includegraphics[scale=\sfact]{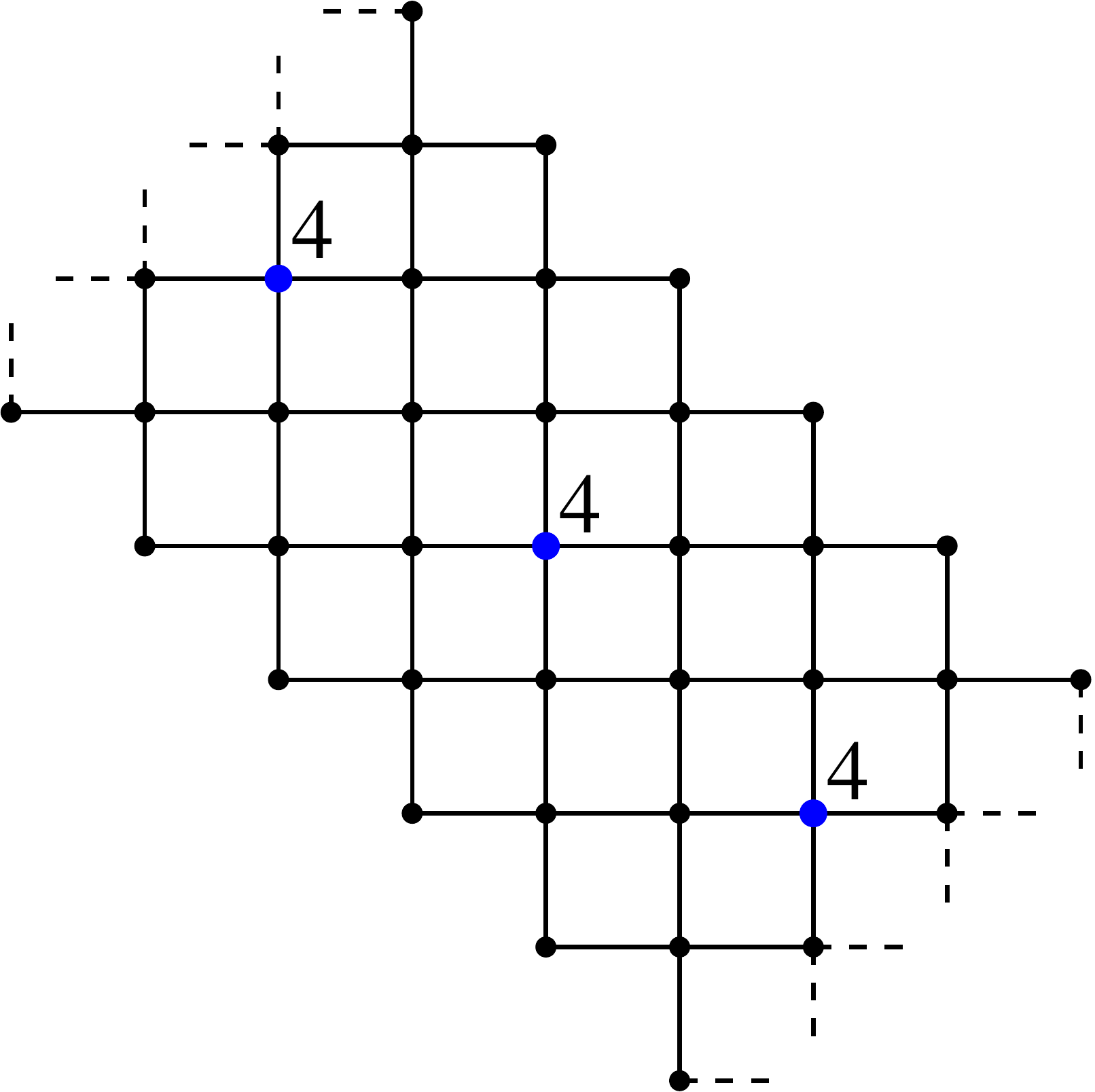} 
\caption{General pattern to make solvable distributions for large graphs with $n+O(1)$ pebbles.}
\label{general7}
\end{figure}

\begin{figure}%[H]
\centering
\scalebox{\sfact}{\input{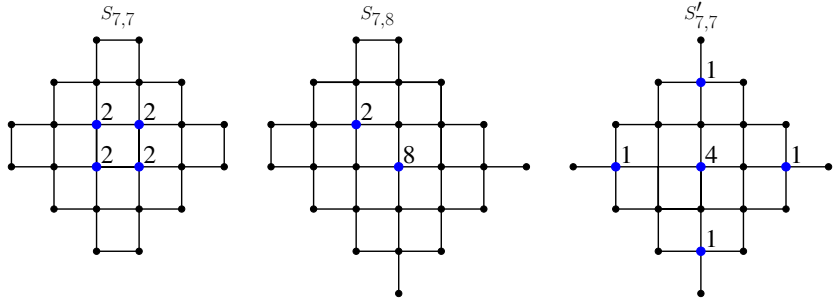}}

\caption{The optimal distributions of small $S_{7,n}$ and $S'_{7,n}$ graphs are pictured. For the $n\leq 6$ cases, see the previous figures.}
\label{S7}
\end{figure}

The proof of the lower bound is slightly different compared to the previous proofs. The main reason is that we do not have a tool for handling distributions with $n+1$ pebbles. $n+1$ pebbles is enough to construct a 2-solvable distribution on the $n$-vertex path, so we can not guarantee a cut which creates a smaller instance of the problem. 

We overcome this difficulty with an indirect assumption. We assume that there is a minimal counterexample with at most $n$ pebbles. In contrast, we show that it is not minimal. 

Unfortunately we cannot apply this idea for values larger than $n+1$. Therefore we cannot prove an exact formula for $\pi_{\opt}(S^*_7)$.

\newpage
\begin{proof}
We prove the upper bound first. For solvable small distributions see Figure \ref{S7}, and for a general pattern see Figure \ref{general7}. 

The optimal distribution for $S'_{7,4k+3}$ is obtained by putting 1 pebble on the vertices neighboring the degree 1 vertices, and 4 pebbles at $k$ vertices as shown in the general pattern.
%by pumping up the distribution of $S'_{7,7}$ with the general pattern.
The solvable distribution of $S'_{7,4k+1}$ is the concatenation of $S'_{7,6}$ and $S'_{7,4(k-2)+3}$, which has $4k+3$ pebbles. Similarly the solvable distribution of $S_{7,4k}$ is the concatenation of $S_{7,5}$ and $S'_{7,4(k-2)+3}$, which has $4k+2$ pebbles. The solvable distribution of $S_{7,4k+1}$ is given by the general pattern with $4k+4$ pebbles. We add two more slashes to this such that the first has two pebbles to obtain a distribution for $S_{7,4k+3}$ with $4k+6$ pebbles.  For $S_{7,4k+2}$ the optimal distribution of $S'_{7,4k+3}$ can be used with $4k+4$ pebbles.  

To prove the lower bound, assume that the statement is not true, and we let $G=S^*_{7,n}$ be the smallest counterexample. This means that there is a solvable distribution on $G$ with $n$ pebbles. By Lemma \ref{lemmacont7} and Lemma \ref{cancutlemma} we can break the graph into two smaller parts, and $P$ induces a solvable distribution on both. One of the graphs has at most as many pebbles as slashes, but this means that we found a smaller counterexample, which is a contradiction.
\end{proof}

\begin{conjecture}
Each solvable distribution which was mentioned in the proof is optimal, which implies the following when $n\geq 10$:
$$\pi_{\opt}(S_{7,n})=\begin{cases}
n+2 &\text{if }n\equiv 0 \mod 2\\
n+3 &\text{if }n\equiv 1 \mod 2\\
%n+2 &\text{if }n\equiv 2 \mod 4\\
%n+3 &\text{if }n\equiv 3 \mod 4
\end{cases}$$ 
$$\pi_{\opt}(S'_{7,n})=\begin{cases}
n+2 &\text{if }n\equiv 1 \mod 4\\
n+1 &\text{if }n\equiv 3 \mod 4.
\end{cases}$$ 
\end{conjecture}

\newpage
\section{Wider staircases}

\begin{question}
What is the optimal pebbling number of $S_{8,n}$? 
\end{question}
We determined the asymptotic behavior when $n \le 7$, but we think that the general behavior of the eight-wide case differs from the seven-wide case. We obtained $\pi_{\opt}(S_{8,8})=11$ by computer search. Unfortunately, even the $n=9$ case requires more computational power than an average PC has.  
We have some partial results:

Let $P$ be an optimal distribution of $S_{m,n}$. Expand $S_{m,n}$ with an additional $m+1$th negative diagonal to obtain $S_{m+1,n}$. Construct $P'$ from $P$ by placing additional pebbles at intersections of the $m$th negative diagonal and every fourth slash starting from the second one. Put a pebble to the last vertex of the $m$th negative diagonal if it has not obtained one yet. For an example see Figure \ref{example8}. 

\begin{figure}[htb]
\centering
\includegraphics[scale=\sfact]{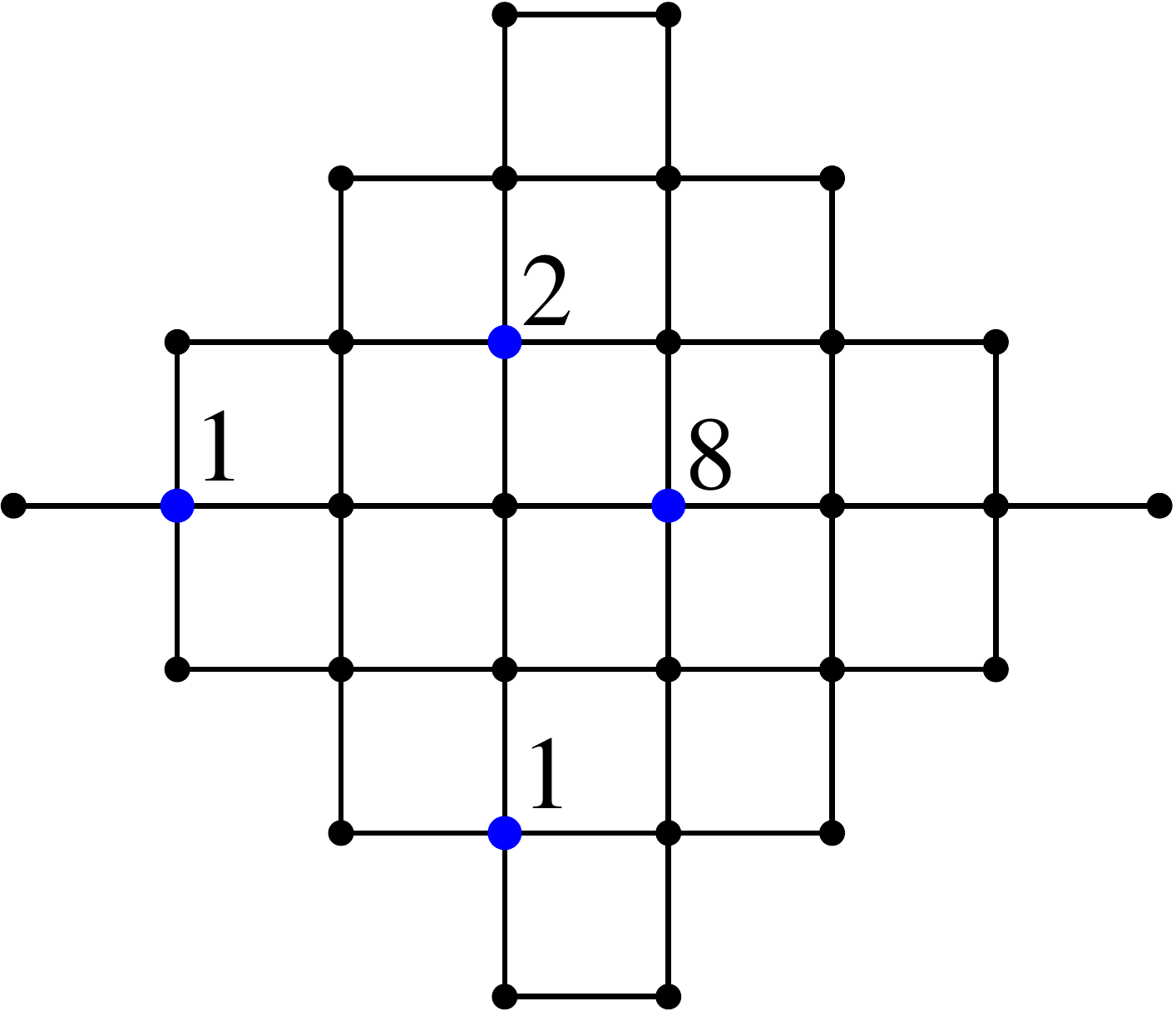} 
\caption{A solvable distribution of $S_{8,8}$. Note that it is not optimal. }
\label{example8}
\end{figure}

\begin{claim}
$P'$ is a solvable distribution of $S_{m+1,n}$.
\end{claim}  

\begin{proof}
The vertices where we placed additional pebbles are $2$-reachable under $P'$. Each vertex of the 8th negative diagonal is adjacent to such a vertex, thus it is reachable.
\end{proof}
We believe that if for each optimal pebbling distribution $P$ for $S_{7,n}$ we form the distribution $P'$, then these distributions will have size equal to the optimal pebbling number of $S_{8,n}$ asymptotically.
\begin{conjecture}
$$  \pi_{\opt}(S_{8,n}) = \frac{5}{4}n + O(1). $$
\end{conjecture}
%C - In my opinion the m=9 case is too optimistic.  Maybe this conjecture and the next should just be stated as asymptotics?

%It is not worth using this method for wider staircases. Solvable distributions for them can be constructed by using the optimal distributions of narrow staircases.% The application of previous method gives bigger distributions for them. 
Finally, it seems that by simply duplicating an optimal distribution for $S_{7,n}$ $k$ times, we can obtain a distribution of asymptotically optimal size for $S_{7k,n}$.
\begin{conjecture}
For all $k\ge 1$, we have
$$  \pi_{\opt}(S_{7k,n}) = kn + O(1). $$
\end{conjecture}

\section*{Acknowledgment}

The researach of Ervin Gy\H{o}ri and Gyula Y. Katona is partially
suported by National Research, Development and Innovation Office
NKFIH, grant K116769. The research of Gyula Y. Katona and L\'aszl\'o F. Papp is partially
suported by National Research, Development and Innovation Office
NKFIH, grant K108947.

\end{document}